\documentclass[12pt,twoside]{article}
\usepackage[english]{babel}
\usepackage[latin1]{inputenc}
\usepackage{amsmath}
\usepackage{amsthm}
\usepackage{amssymb,amsfonts}
\usepackage{comment}
\usepackage{tikz-cd}
\usepackage{tikz}
 \usetikzlibrary {shapes.geometric} 
\usepackage{orcidlink}
\usepackage{stmaryrd}
\usepackage{pst-node}
\usepackage{caption}
\usepackage{array}
\usepackage[all]{xy}

\usepackage{graphicx}                   
\usepackage{hyperref}

\newcommand{\CP}[1]{\mathbb{C}P^{#1}}      

\begin{document}
\raggedbottom
\pagestyle{myheadings}
\title
{\textsc{Rank three representations of Painlev\'e systems: II. de Rham structure, Fourier--Laplace transformation}}

\author{by \textsc{Mikl\'os Eper}\footnote{\textsl{Department of Algebra and Geometry, Institute of Mathematics, Faculty of Natural Sciences, Budapest University of Technology and
Economics, M\H uegyetem rkp. 3., Budapest H-1111, Hungary}, e-mail: \href{mailto:epermiklos@gmail.com}{epermiklos@gmail.com}} and \textsc{Szil\'ard Szab\'o}\footnote{Institute of Mathematics, Faculty of Science, E\"otv\"os Lor\'and University, P\'azm\'any P\'eter s\'et\'any 1/C, Budapest, Hungary, H-1117; HUN-REN Alfr\'ed R\'enyi Institute of Mathematics,
Re\'altanoda utca 13-15., Budapest 1053, Hungary, e-mails:\href{mailto:szilard.szabo@ttk.elte.hu}{szilard.szabo@ttk.elte.hu} and \href{mailto:szabo.szilard@renyi.hu}{szabo.szilard@renyi.hu}}
\orcidlink{https://orcid.org/my-orcid?orcid=0009-0008-7435-2021}}

\maketitle


\begin{abstract}
We use formal microlocalization to describe the Fourier--Laplace transformation between rank 3 and rank 2 $\mathcal{D}$-module representations of Painlev\'e systems. We conclude the existence of biregular morphism between the corresponding de Rham complex structures. 
\end{abstract}

\newtheorem{theorem}{Theorem}[section]
\newtheorem{corollary}[theorem]{Corollary}
\newtheorem{conjecture}{Conjecture}[section]
\newtheorem{lemma}[theorem]{Lemma}
\newtheorem{exmple}[theorem]{Example}
\newtheorem{defn}[theorem]{Definition}
\newtheorem{prop}[theorem]{Proposition}
\newtheorem{rmrk}[theorem]{Remark}
\newtheorem{claim}[theorem]{Claim}
\newtheorem{assumption}[theorem]{Assumption}

\newenvironment{definition}{\begin{defn}\normalfont}{\end{defn}}
\newenvironment{remark}{\begin{rmrk}\normalfont}{\end{rmrk}}
\newenvironment{example}{\begin{example}\normalfont}{\end{example}}
\newenvironment{acknowledgement}{{\bf Acknowledgement:}}

\newcommand\restr[2]{{
  \left.\kern-\nulldelimiterspace 
  #1 
  \vphantom{\big|} 
  \right|_{#2} 
  }}


\section{Introduction}

Recently~\cite{ESz2}, we have studied the wild character varieties associated to the rank $3$ Lax representations of the Painlev\'e equations found by Joshi, Kitaev and Treharne (JKT)~\cite{JKTI},~\cite{JKTII}. 
The aim of the present article is to offer a comprehensive description of the same correspondence, this time between the corresponding de Rham moduli spaces of irregular meromorphic connections, in terms of Fourier--Laplace transformation of $\mathcal{D}$-modules. 
Our main result can be summarized as: 
\begin{theorem}[Proposition~\ref{prop:dRisomorphism}]\label{thm:FLisom}
Fourier--Laplace transformation $\mathcal{F}$ of $\mathcal{D}$-modules induces a biregular morphism between the moduli spaces of rank $2$ irregular connections associated to the Painlev\'e equations and the corresponding JKT moduli spaces of irregular connections of rank $3$.  
\end{theorem}
To state a precise version, we will need to write local forms of the rank $3$ irregular connections in Sections~\ref{sec:dR}. 
Fourier--Laplace transformation of holonomic $\mathcal{D}$-modules, as an equivalence of categories, has a long history, dating back at least to the seminal work of Malgrange~\cite{Malgr}, but the special case of irregularity of slope purely $1$ at infinity and a logarithmic singularity at $0$ had been studied already by Balser, Jurkat and Lutz~\cite{BJL}. 
Fourier--Laplace transformation in the Painlev\'e VI case has been first analyzed by Harnad~\cite[Section~3.c]{Har} from a Hamiltonian perspective and using loop algebras; the properties of the transformation have been worked out by Mazzocco~\cite{Maz}. 
In~\cite{Sz_Laplace}, the second author studied Fourier--Laplace transformation as an algebraic map between suitable de Rham moduli spaces in the special case when the singularities away from $\infty$ are regular and the irregular singularity at $\infty$ is of the simplest possible type (unramified of Katz invariant $1$).
This setup applies directly to one of the cases that we consider here (Painlev\'e VI), and the approach carries over too, up to some technical details. 
Our approach to finding the irregular types of the transformed objects relies on a systematic use of formal microlocalization. 

We note that in~\cite[Theorem~1.1]{Boa5}, Boalch described certain isomorphisms between open subsets of moduli spaces of irregular connections defined over the trivial bundle given by different "readings" of so-called supernova graphs. 
One important instance of such isomorphisms is precisely given by Fourier--Laplace transformation. 
In~\cite[Section~4]{Dou}, the isomorphisms of Theorem~\ref{thm:FLisom} are analyzed from the perspective of supernova graphs. 
Let us point out two differences in the results we find: in the case Painlev\'e V, we do find a rank $3$ representative, while the one of higher rank found by Dou\c{c}ot is of rank $4$; and in the doubly degenerate Painlev\'e III case (corresponding to $\widetilde{D}_8$) we do not find a  rank $3$ representative, whereas~\cite{Dou} does find one. 
Currently, we do not have an explanation for these differences. 
We also note that Alameddine and Marchal~\cite{AM} discusses the Painlev\'e IV case of this isomorphism in great detail from the isomonodromy perspective. 

The method of our proof is expected to generalize to showing that Fourier--Laplace transformation is an algebraic isomorphism between suitable non-abelian Hodge moduli spaces in higher generality, too. 

\section{Preliminaries}

\subsection{Meromorphic connections, $\mathcal{D}$-modules}\label{sec:connection}

Let $X = \operatorname{Spec} \mathbb{C}[z] \cup \operatorname{Spec} \mathbb{C}[w]$ be the projective line $\mathbb{C}P^1$, where $z w = 1$. 
Let $\mathcal{D}_X$ be the sheaf of analytic differential operators on $X$. 
It is the sheaf of noncommutative $\mathcal{O}_X$-algebras that is generated over $\operatorname{Spec}\mathbb{C}[z]$ by $\partial_z$, 
subject to the relation $[\partial_z, z] = 1$. 
It has a similar description over $\operatorname{Spec}\mathbb{C}[w]$, with the relationship $w\partial_w = - z\partial_z$ between its generators. 
Similarly, one can define the noncommutative algebra $\mathbb{C}[t] \langle \partial_t \rangle$ of polynomial differential operators on $X$. 
A left module $\mathbb{M}$ over $\mathcal{D}_X$ is said to be holonomic if it is finitely generated and torsion. 
In the sequel, $\mathbb{M}$ will always denote a holonomic left $\mathcal{D}_X$-module. 
The rank of $\mathbb{M}$ is then defined as 
\[
    \operatorname{rk} \mathbb{M} = \dim_{\mathbb{C}(z)} \mathbb{C}(z) \otimes_{\mathbb{C}[z]} \mathbb{M} .
\]
Moreover, there exists a maximal Zariski open subset $U\subset X$ such that $\mathbb{M}|_U$ is finitely generated over $\mathcal{O}_U$. 
The set of singular points of $\mathbb{M}$ is then defined to be $X\setminus U$ and is denoted by $C = \operatorname{Sing}(\mathbb{M})$. 
We denote by $j$ the inclusion $U\to X$. 

\subsection{Formal theory}

We will focus on the formal structure of $\mathbb{M}$ at $\operatorname{Sing}(\mathbb{M})$. 
For any $c\in \operatorname{Sing}(\mathbb{M})$ we pick a local holomorphic coordinate $t$ of $X$. 
For $c\in \operatorname{Spec}\mathbb{C}[z]$, we may take $t = z-c$, and for $c=\infty$, we take $t=w$. 
A special role is played by regular singular modules, which are by definition the ones that admit a fundamental system of at most polynomial growth in $|t|$ as $t\to 0$ within a sector of finite opening. 
Next, we consider 
\[
    \mathbb{M}_c =  \mathbb{C} \llbracket t \rrbracket \langle \partial_t \rangle \otimes_{\mathbb{C}[t] \langle \partial_t \rangle} \mathbb{M} .
\]
The Hukuhara--Levelt--Turrittin theorem~\cite{Huk},~\cite{Lev},~\cite{Tur}  then states that there exists some positive integer $d$ and a finite Galois extension 
\begin{equation}\label{eq:extension}
    K = \mathbb{C} (\!( u , t )\!)/(u^d - t)     
\end{equation}
of $\mathbb{C} (\!( t )\!)$  such that 
\begin{equation}\label{eq:HLT}
    K\otimes_{\mathbb{C}[t]} \mathbb{M} \cong \bigoplus_{q\in K/\mathbb{C}\llbracket u \rrbracket} 
(\mathbb{C}\llbracket u \rrbracket ,\operatorname{d}+\operatorname{d}\! q) 
\otimes_{\mathbb{C}\llbracket u \rrbracket} \mathcal{F}_{q} .
\end{equation}
Here $\operatorname{d}$ stands to denote the trivial connection and the exterior differentiation operator with respect to $u$ and $\mathcal{F}_{q}$ is a free $\mathbb{C}\llbracket t \rrbracket$-module with regular singular connection. 
For all but finitely many values of $q$ the rank of $\mathcal{F}_{q}$ is $0$. 
The decomposition is unique up to permutations, gauge transformations and further field extensions. 
We define the slope of $q$ as 
\[
    \lambda (q) = \frac{\deg_u \! q}d; 
\]
it is clear that this quantity remains unchanged under further Galois extensions $L\vert K$. 
We say that $\mathbb{M}_c$ is pure of slope $\lambda$ if all its nontrivial summands $q$ are of the same slope $\lambda (q) = \lambda$. 
We call the maximal slope appearing in the decomposition the Katz-invariant of $\mathbb{M}_c$. 
We denote by $\mathbb{M}^{<\lambda}$ the direct sum of its pure constituents  of slopes strictly less than $\lambda$, and similarly for $\mathbb{M}^{\leq \lambda}, \mathbb{M}^{> \lambda}, \mathbb{M}^{\geq \lambda}$. 
Finally, we define the Swan conductor (or irregularity) of $\mathbb{M}_c$ as 
\[
    \operatorname{Sw} (\mathbb{M}_c ) = \sum_q \lambda (q) \operatorname{rk}(\mathcal{F}_{q}). 
\]
It turns out~\cite[Th\'eor\`eme~(2.1.2.6)]{Lau} that $\operatorname{Sw} (\mathbb{M}_c ) \in \mathbb{N}$.  
As a matter of fact, each individual summand is equal to the height of the side corresponding to $q$ in the Newton polygon of $\mathbb{M}_c$ with respect to the variables $t, t^{-1}\partial_t$, see~\cite[Chapitre~III]{Malgr}. 



\subsection{Geometric Fourier--Laplace transformation}

See~\cite[Chapitre~VI]{Malgr},~\cite{Lau},~\cite{BE},~\cite{GL},~\cite{Sab},~\cite{Sz_Laplace} for more details on the material of this section. 
(Note that~\cite{Lau} does not exactly concern the $\mathcal{D}$-module theoretic Fourier transformation that we need, but rather a very closely related one, namely the one in the derived category of $\ell$-adic sheaves, but the intuition it provides is useful.) 

Let $\widehat{X}$ be another copy of $\CP1$: 
\[
\widehat{X} =\operatorname{Spec}\mathbb{C}[\hat{z} ] \cup \operatorname{Spec}\mathbb{C}[\hat{w} ]
\]
with $\hat{w} = \hat{z}^{-1}$. 
The variable $\hat{z}$ will be thought of as the dual variable of $z$. 
We denote the projection morphisms by 
\begin{equation}\label{eq:projections}
        p\colon X\times \widehat{X}\to X, \qquad \hat{p}\colon X\times \widehat{X}\to  \widehat{X}. 
\end{equation}
Let us denote by $\psi$ the sheaf of $\mathcal{D}_{X\times \widehat{X}}$-modules that is the free $\mathcal{O}_{X\times \widehat{X}}$-module of rank $1$, endowed with the connection 
\[
    \operatorname{d}+\operatorname{d} (z \hat{z}) = \operatorname{d}+\hat{z} \operatorname{d}\! z + z \operatorname{d}\! \hat{z}. 
\]
We call $\psi$ the \emph{kernel $\mathcal{D}$-module} of the transformation. 
For convenience, we use the notation $\operatorname{Sing}(\mathbb{M}) = C$. 
A minimal extension $\mathbb{M}_{\operatorname{min}}$ of $\mathbb{M}$ over $C$ is defined by the properties: 
\begin{enumerate}
    \item $\mathbb{C}(z) \otimes_{\mathbb{C}[z]} \mathbb{M}_{\operatorname{min}} \cong \mathbb{C}(z) \otimes_{\mathbb{C}[z]} \mathbb{M}$ (it is an extension), and
    \item $\mathbb{M}_{\operatorname{min}}$ admits no quotient or submodules supported in dimension $0$ (it is minimal).
\end{enumerate}
There exists a unique minimal extension; see~\cite{Sab2}. 
We then define the Fourier--Laplace transform of $\mathbb{M}$ by 
\begin{equation}\label{eq:FL}
        \widehat{\mathbb{M}} = \mathcal{F}(\mathbb{M} ) = R \hat{p}_* ( p^* \mathbb{M}_{\operatorname{min}}  \otimes \psi ) [1] 
\end{equation}
where, as usual, $[1]$ stands for the shift functor and
\begin{align*}
    p^* &\colon D^b_{\operatorname{hol}}(\mathcal{D}_X) \to D^b_{\operatorname{hol}}(\mathcal{D}_{X\times \widehat{X}}) \\
    R\hat{p}_* &\colon D^b_{\operatorname{hol}}(\mathcal{D}_{X\times \widehat{X}}) \to D^b_{\operatorname{hol}}(\mathcal{D}_{\widehat{X}})
\end{align*}
are the pull-back and right derived direct image functors respectively, between derived categories of bounded complexes of holonomic $\mathcal{D}$-modules.
We will denote by $\widehat{C}$ the set of singular points of~\eqref{eq:FL}. 

\subsection{Local Fourier--Laplace transformation}
In order to understand the contribution of the singularities of $\mathbb{M}$ at its singular points $C$ to the singularities of its Fourier transform $\widehat{\mathbb{M}}$, Laumon defined~\cite[D\'efinition~(2.4.2.3)]{Lau} local versions $\mathcal{F}^{0,\widehat{\infty}}, \mathcal{F}^{\infty,\hat{0}},\mathcal{F}^{\infty,\widehat{\infty}}$ of the functor $\mathcal{F}$ for $\ell$-adic sheaves. 
They were subsequently translated into the category of $\mathcal{D}$-modules by Malgrange~\cite[Section~VII.3]{Malgr}. 
More recent accounts of local Fourier transformation of $\mathcal{D}$-modules have been proposed by~\cite[Proposition-Definition~3.5, Definitions~3.8,~3.11]{BE},~\cite[p.~750]{GL}. 
We will now describe these local Fourier transformations. 
There are three essentially equivalent descriptions: via vanishing cycle functors~\cite{Lau}, formal microlocalization~\cite{GL} or Deligne--Malgrange's notion of a good pair of lattices~\cite{BE}. 
We adopt here the third point of view, and the second one in the next section. 
A \textit{pair of good lattices} for a $\mathcal{D}_X$-module $\mathbb{M}$ induced by a meromorphic connection $\nabla$ over $(X,C)$ is a pair of vector bundles $\mathcal{V}, \mathcal{W}$ over $X$ satisfying 
\begin{enumerate}
    \item $\mathcal{V}\subset \mathcal{W} \subset j_* \mathbb{M}$
    \item $\nabla ( \mathcal{V} ) \subset K_X(C) \otimes \mathcal{W}$
    \item For any effective divisor $D$ supported on $C$, the inclusion 
    \[
    \left( \mathcal{V} \xrightarrow{\nabla} K_X(C) \otimes \mathcal{W} \right) \to \left( \mathcal{V}(D) \xrightarrow{\nabla} K_X(C) \otimes \mathcal{W}(D)\right)
    \]
    is a quasi-isomorphism.
\end{enumerate}
Good pairs of lattices exist~\cite{Del}. 

Now, assume that $\mathbb{M}$ has a (possibly irregular) singularity at $c$; we will take $c=0$ for simplicity. 
Let us restrict $\mathbb{M}$ to a sufficiently small disk centered at $0$ (containing no other singular point) and take any regular singular extension of $\mathbb{M}$ to $z=\infty$. 
Let $\mathcal{V}, \mathcal{W}$ be a pair of good lattices for this extension. 
We then set
\[
    \mathcal{F}^{0,\widehat{\infty}}(\mathbb{M}) = \operatorname{coker}\left(  \mathcal{V}\otimes_{\mathbb{C}[z]} \mathbb{C} \llbracket z, \hat{w} \rrbracket \xrightarrow{\hat{w}\partial_z + 1} z^{-1} \mathcal{W}\otimes_{\mathbb{C}[z]} \mathbb{C} \llbracket z, \hat{w} \rrbracket \right), 
\]
where we use the shorthand $\partial_z$ for $\nabla_{\partial_z}$. 
The operator $\partial_z + \partial_{\hat{z}} + z + \hat{z}$ induces a natural structure of $\mathbb{C} [ \hat{w} ] \langle \partial_{\hat{w}} \rangle$-module on $\mathcal{F}^{0,\widehat{\infty}}(\mathbb{M})$ (recall that $\hat{w} = \hat{z}^{-1}$). 
Similar formulas define $\mathcal{F}^{\infty,\hat{0}},\mathcal{F}^{\infty,\widehat{\infty}}$. 
The properties of these local Fourier transformations can be summarized as follows.
\begin{theorem}\label{thm:LocalFourier}~\cite[Th\'eor\`eme~(2.4.3)]{Lau}~\cite[Proposition~3.14]{BE}
The functors $\mathcal{F}^{0,\widehat{\infty}},\mathcal{F}^{\infty,\hat{0}},\mathcal{F}^{\infty,\widehat{\infty}}$ are exact. Furthermore, 
\begin{enumerate}
    \item \label{thm:LocalFourier1}
    we have $\operatorname{rk}(\mathcal{F}^{0,\widehat{\infty}}(\mathbb{M})) = \operatorname{rk}(\mathbb{M}) + \operatorname{Sw}(\mathbb{M}_0), \quad \operatorname{Sw}(\mathcal{F}^{0,\widehat{\infty}}(\mathbb{M})_{\widehat{\infty}}) = \operatorname{Sw}(\mathbb{M}_0)$; 
    \item \label{thm:LocalFourier2}
    if $\mathbb{M}_{\infty}$ is pure of slope $\lambda <1$, then we have 
    \[
    \operatorname{rk}(\mathcal{F}^{\infty,\hat{0}}(\mathbb{M})) = \operatorname{rk}(\mathbb{M}_{\infty}) - \operatorname{Sw}(\mathbb{M}_{\infty}), \quad  \operatorname{Sw}(\mathcal{F}^{\infty,\hat{0}}(\mathbb{M})_{\hat{0}}) = \operatorname{Sw}(\mathbb{M}_{\infty}),
    \]
    and if $\mathbb{M}_{\infty}$ is pure of slope $\lambda \geq 1$, then we have 
    \[
    \mathcal{F}^{\infty,\hat{0}}(\mathbb{M}) = 0; 
    \]
    \item \label{thm:LocalFourier3}
    if $\mathbb{M}_{\infty}$ is pure of slope $\lambda > 1$, then we have 
    \[
    \operatorname{rk}(\mathcal{F}^{\infty,\widehat{\infty}}(\mathbb{M})) = \operatorname{Sw}(\mathbb{M}_{\infty}) - \operatorname{rk}(\mathbb{M}_{\infty}), \quad \operatorname{Sw}(\mathcal{F}^{\infty,\widehat{\infty}}(\mathbb{M})_{\widehat{\infty}}) = \operatorname{Sw}(\mathbb{M}_{\infty})
    \]
    and if $\mathbb{M}_{\infty}$ is pure of slope $\lambda \leq 1$, then we have 
    \[
    \mathcal{F}^{\infty,\widehat{\infty}}(\mathbb{M}) = 0.
    \]
\end{enumerate}
\end{theorem}

Regarding the summands of slope $1$, we have the following special cases of the formal stationary phase formula. 
Let us denote the set of singularities of $\widehat{\mathbb{M}}$ by $\widehat{C}$.  
\begin{theorem}\label{thm:StationaryPhase}~\cite[Th\'eor\`eme~VIII.3.3]{Malgr},~\cite[Lemme~(2.4.2.1)]{Lau}
Let $\mathbb{M}$ be a holonomic $\mathcal{D}_X$-module with singular set $C\subseteq \{ 0, \infty \}$. 
We assume that if $0\in C$ then it is a regular singularity, and $\infty$ is an irregular singularity. 
Then,  for all nontrivial summands of slope $1$ in the HLT-decomposition~\eqref{eq:HLT} of $\mathbb{M}_{\infty}$, i.e. of the form $Q_j = q_j w^{-1} + o(w^{-1})$ with $q_j\in \mathbb{C}$, we have $-q_j \in \widehat{C}$.
Moreover, $\widehat{\mathbb{M}}$ has a regular singularity at $-q_j$ whose monodromy around $-q_j$ is equal to the direct sum of the monodromy of $\mathcal{F}_{Q_j}$ around $\infty$ with the identity matrix of appropriate rank $n_j = \operatorname{rk} (\widehat{\mathbb{M}}) - \operatorname{rk} (\mathcal{F}_{Q_j})$. 
Furthermore, if $q\in \widehat{C}$ then either $q\in \{ \hat{0}, \widehat{\infty} \}$ or $q=-q_j$ for some summand $Q_j = q_j w^{-1} + o(w^{-1})$ in the HLT-decomposition of $\mathbb{M}_{\infty}$. 
\end{theorem}

Equivalent formulation of the Fourier--Laplace transformation comes from the following isomorphism of Weyl-algebras $\mathbb{C}[z]\langle \partial_z\rangle$ and $\mathbb{C}[\hat{z}]\langle \partial_{\hat{z}}\rangle$, called the Fourier--Laplace anti-involution:
\begin{equation}\label{eq:inv}
\begin{aligned}
    z & \mapsto-\partial_{\hat{z}} \\
    \partial_z & \mapsto\hat{z}
    \end{aligned}
\end{equation}
This map equips the $\mathbb{C}[z]\langle \partial_z\rangle$-module $\mathbb{M}$ with a $\mathbb{C}[\hat{z}]\langle \partial_{\hat{z}}\rangle$-module structure, denoted by $\widehat{\mathbb{M}}$, and vice versa.

\subsection{Formal microlocalization and formal stationary phase}\label{sec:microloc}

The method of formal microlocalization, related to Theorem~\ref{thm:StationaryPhase}, offers an explicit understanding of the formal types of $\widehat{\mathbb{M}}$ in terms of those of $\mathbb{M}$. 
For reference, see~\cite[Section~V.3]{Sab_Frobenius} and \cite{GL}. 

The ring $\mathcal{E}_0$ of formal microdifferential operators at $0$ consists of all formal series of the form 
\begin{equation}\label{eq:micro}
    \sum_{i\geq i_0} a_i (z) \hat{w}^i = \sum_{i\leq - i_0} a_{-i} (z) \hat{z}^i
\end{equation}
for some $i_0 \in \mathbb{Z}$, where $a_i \in \mathbb{C}\llbracket z \rrbracket$. 
The (noncommutative!) product operation on $\mathcal{E}_0$ is given by 
\[
    \hat{w} \cdot a(z) = \sum_{k=0}^{\infty} (-1)^k a^{(k)} (z) \hat{w}^{k+1}. 
\]
This is compatible with (formal) Laplace transformation because applying $\hat{w}^{-1} = \hat{z}=\partial_z$ to both sides and using the Leibniz rule we get 
\[
    a(z) = \sum_{k=0}^{\infty} (-1)^k (a^{(k+1)} (z) \hat{w}^{k+1} + a^{(k)} (z) \hat{w}^k ),
\]
which is a formal identity. 
The ring $\mathcal{E}_0$ is a right-left $\mathbb{C} [z ] \langle \partial_z \rangle$-bimodule. 
The formal microlocalized module of a $\mathbb{C} [z ] \langle \partial_z \rangle$-module $\mathbb{M}$ at $0$ is the left $\mathcal{E}_0$-module 
\[
    \mathbb{M}^{\mu}_0 = \mathcal{E}_0 \otimes_{\mathbb{C} [z ] \langle \partial_z \rangle} \mathbb{M}.
\]
Crucially, it turns out that $\mathbb{M}^{\mu}_0$ is supported over the closed point $(z)\in \operatorname{Specf} \mathbb{C} \llbracket z \rrbracket$. 
There is a similar construction for the microlocalized module $\mathbb{M}^{\mu}_{\infty}$ of $\mathbb{M}$ at $\infty$, supported at $(w)\in \operatorname{Specf} \mathbb{C} \llbracket w \rrbracket$, with the ring $\mathcal{E}_{\infty}$, consisting of formal series of the form
\begin{equation}\label{eq:micro2}
    \sum_{i\geq i_0} a_i (z^{-1}) \hat{w}^i = \sum_{i\leq - i_0} a_{-i} (z^{-1}) \hat{z}^i
\end{equation}
for some $i_0\in\mathbb{Z}$, and $a_i\in\mathbb{C}\llbracket z^{-1} \rrbracket$.
We set 
\[
    \mathbb{C} \llbracket \hat{w} \rrbracket \langle \partial_{\hat{w}} \rangle = \mathbb{C} [ \hat{w} ] \langle \partial_{\hat{w}} \rangle \otimes_{\mathbb{C} [ \hat{w} ]} \mathbb{C} \llbracket \hat{w} \rrbracket .
\]
\begin{theorem}~\cite[Formal~stationary~phase]{GL}\label{thm:Fsp}
    There exists an isomorphism of $\mathbb{C} \llbracket \hat{w} \rrbracket \langle \partial_{\hat{w}} \rangle$-modules 
    \[
    \widehat{\mathbb{M}}_{\widehat{\infty}} \otimes_{\mathbb{C} [\hat{w}]\langle \partial_{\hat{w}} \rangle} \mathbb{C} \llbracket \hat{w} \rrbracket \langle \partial_{\hat{w}} \rangle \xrightarrow{\sim} \bigoplus_{c\in C \cup \{ \infty \}} \mathbb{M}^{\mu}_c. 
    \]
\end{theorem}
Applying this result in specific cases provides more refined information than Theorems~\ref{thm:LocalFourier} and~\ref{thm:StationaryPhase}. 
Namely, in addition to the numerical invariants $\operatorname{rk}, \operatorname{Sw}$, we also get the irregular types. 
We will work out such examples in Section~\ref{sec:dR}. 

\subsection{Local forms}

We will be interested in the special cases of this general theory when $\mathbb{M}$ is the $\mathcal{D}$-module over $\mathbb{C}(z)$ associated to a  meromorphic connection $(\mathcal{E},\nabla)$ of one of six specific local forms.

Consider a meromorphic connection $\nabla$ on a holomorphic vector bundle $\mathcal{E}\rightarrow X$ of rank three. 
We require $\nabla$ to have prescribed singularities at the points of $c_i\in C$ with irregular part
\[
Q_i=\sum_{k=-m_i}^{-1}A_{k-1}(z-c_i)^k,
\]
in some holomorphic coordinate chart $z-c_i$ centered at $c_i$, and local trivialization of $\mathcal{E}$. 
The eigenvalues $q_i$ of $Q_i$ are called the de Rham irregular types. 
Then the de Rham local form of $\nabla$ at $c_i$ is defined as
\begin{equation}\label{eq:locform}
\nabla=\operatorname{d}+\operatorname{d}\! Q_i+\Lambda_i\frac{\operatorname{d}\! z}{z-c_i}+O(1)\operatorname{d}\! z,
\end{equation}
where $\Lambda_i$ is some constant matrix called the residue, and $O(1)$ refers to the holomorphic part.
Let us give the specific divisors $D$ and irregular parts $Q_i$ that we will consider in the rest of the paper. 
In the first three cases JKTVI, JKTV, JKTIVa we have $D=\{0\}+2\cdot \{\infty\}$.
In these cases, the eigenvalues of $\operatorname{Res}_0 \nabla$ are fixed generic values. 
In the remaining cases JKTIVb, JKTII, JKTI we have $D=3\cdot \{\infty\}$. (Here $JKT*$ refers to the system described in \cite{JKTI} and \cite{JKTII}.) 
As for the normal forms, they are listed in the following formulas. 

\paragraph{JKTVI}
\begin{equation}
\label{connection1}\tag{JKTVI-dR}
    \nabla=\operatorname{d}+\left[
    \begin{pmatrix}
        a_0 & 0 & 0 \\
        0 & a_1 & 0 \\
        0 & 0 & a_2
    \end{pmatrix}w^{-2}+\begin{pmatrix}
        b_0 & 0 & 0 \\
        0 & b_1 & 0 \\
        0 & 0 & b_2
    \end{pmatrix}w^{-1}+O(1)
    \right]\otimes\textrm{d}w,
\end{equation}
where 
\[
b_0+b_1+b_2=\operatorname{Tr Res}_{0} \nabla
\]
because of the residue theorem. 

\paragraph{JKTV}
\begin{equation}
\label{connection2}\tag{JKTV-dR}
    \nabla=\operatorname{d}+\left[
    \begin{pmatrix}
        a_0 & 1 & 0 \\
        0 & a_0 & 0 \\
        0 & 0 & a_1
    \end{pmatrix}w^{-2}+\begin{pmatrix}
        0 & 0 & 0 \\
        b_0 & b_1 & 0 \\
        0 & 0 & b_2
    \end{pmatrix}w^{-1}+O(1)
    \right]\otimes\textrm{d}w,
\end{equation}
where 
\[
b_1+b_2=\operatorname{Tr Res}_{0} \nabla
\] 
because of the residue theorem.

\paragraph{JKTIVa}
\begin{equation}
\label{connection3}\tag{JKTIVa-dR}
    \nabla=\operatorname{d}+\left[
    \begin{pmatrix}
        a_0 & 1 & 0 \\
        0 & a_0 & 1 \\
        0 & 0 & a_0
    \end{pmatrix}w^{-2}+\begin{pmatrix}
        0 & 0 & 0 \\
        0 & 0 & 0 \\
        b_0 & b_1 & b_2
    \end{pmatrix}w^{-1}+O(1)
    \right]\otimes\textrm{d}w,
\end{equation}
where 
$b_2=\operatorname{Tr Res}_{0} \nabla$
because of the residue theorem.

\paragraph{JKTIVb}
\begin{equation}
\label{connection4}\tag{JKTIVb-dR}
    \nabla=\operatorname{d}+\left[
    \begin{pmatrix}
        a_0 & 0 & 0 \\
        0 & a_1 & 0 \\
        0 & 0 & a_2
    \end{pmatrix}w^{-3}+\begin{pmatrix}
        b_0 & 0 & 0 \\
        0 & b_1 & 0 \\
        0 & 0 & b_2
    \end{pmatrix}w^{-2}+\begin{pmatrix}
        c_0 & 0 & 0 \\
        0 & c_1 & 0 \\
        0 & 0 & c_2
    \end{pmatrix}w^{-1}+O(1)
    \right]\otimes\textrm{d}w,
\end{equation}
where $c_0+c_1+c_2=0$, because of the residue theorem.

\paragraph{JKTII}
\begin{equation}
\label{connection5}\tag{JKTII-dR}
    \nabla=\operatorname{d}+\left[
    \begin{pmatrix}
        a_0 & 1 & 0 \\
        0 & a_0 & 0 \\
        0 & 0 & a_1
    \end{pmatrix}w^{-3}+\begin{pmatrix}
        0 & 0 & 0 \\
        b_0 & b_1 & 0 \\
        0 & 0 & b_2
    \end{pmatrix}w^{-2}+\begin{pmatrix}
        0 & 0 & 0 \\
        c_0 & c_1 & 0 \\
        0 & 0 & c_2
    \end{pmatrix}w^{-1}+O(1)
    \right]\otimes\textrm{d}w,
\end{equation}
where $c_1+c_2=0$, because of the residue theorem.

\paragraph{JKTI}
\begin{equation}
\label{connection6}\tag{JKTI-dR}
    \nabla=\operatorname{d}+\left[
    \begin{pmatrix}
        a_0 & 1 & 0 \\
        0 & a_0 & 1 \\
        0 & 0 & a_0
    \end{pmatrix}w^{-3}+\begin{pmatrix}
        0 & 0 & 0 \\
        0 & 0 & 0 \\
        b_0 & b_1 & b_2
    \end{pmatrix}w^{-2}+\begin{pmatrix}
        0 & 0 & 0 \\
        0 & 0 & 0 \\
        c_0 & c_1 & c_2
    \end{pmatrix}w^{-1}+O(1)
    \right]\otimes\textrm{d}w,
\end{equation}
where $c_2=0$ because of the residue theorem.

\begin{theorem}~\cite{BB}\label{thm:moduli}
    For generically fixed numerical invariants $a_i, b_i, c_i$ (as appropriate), there exist moduli spaces ${\mathcal{M}}_{dR}^{JKT*}$ of gauge equivalence classes of $\beta$-stable integrable connections with irregular singularities of fixed irregular types and residues. 
    Moreover, they are complete and carry natural holomorphic symplectic structures. 
\end{theorem}

This is a twisted special case of~\cite[Theorem~0.2]{BB}. 
We give more details in our companion paper~\cite[Theorem~2.1]{ESz4}. 
This statement does not play a crucial role in the computations of Section~\ref{sec:dR}, but it is needed to state Proposition~\ref{prop:dRisomorphism}.
Regarding the dimension, note that the constructions of the diffeomorphic wild character varieties as affine cubic surfaces over $\mathbb{C}$ are given in our companion paper \cite{ESz2}.

These local forms provide us some crucial data about $\mathbb{M}_{\infty}$, namely the slopes in $\mathbb{M}_{\infty}$. These are summarized in the following table for generic values of $a_i,b_i,c_i$.

\begin{center}
\begin{tabular}{ c|c|c|c|c|c|c } 
 case & $JKTVI$ & $JKTV$ & $JKTIVa$ & $JKTIVb$ & $JKTII$ & $JKTI$ \\ 
 \hline
 slopes in $\mathbb{M}_{\infty}$ & $1,1,1$ & $1/2,1$ & $2/3$ & $2,2,2$ & $3/2,2$ & $5/3$ \\ 
\end{tabular}
\captionof{table}{}
\label{tab:slopes}
\end{center}

\section{de Rham spaces}\label{sec:dR}

We show how Fourier--Laplace transformation (with suitable choices of parameters) maps the moduli spaces of irregular connections that we consider bijectively to the well-known Painlev\'e phase spaces. 
Once this is done, just as in~\cite{Sz_Plancherel}, it follows from the algebraic definition~\eqref{eq:FL} of Fourier--Laplace transformation, that: 
\begin{prop}\label{prop:dRisomorphism}
    Fourier--Laplace transformation establishes an algebraic isomorphism with respect to the de Rham complex structures between each $\mathcal{M}_{\textrm{dR}}^{JKT*}$ and the corresponding $\mathcal{M}_{\textrm{dR}}^{P*}$. 
\end{prop}
This is a more detailed statement of Theorem~\ref{thm:FLisom}. 
It also provides interesting examples of the isomorphisms of~\cite{Boa5}. To provide this isomorphism, we will need to investigate the Fourier-Laplace transformation only in one direction (from $JKT$ to Painlev\'e), because of the following well known statement.

\begin{prop}
    For the square of the Fourier-Laplace transform $\mathcal{F}$ of $\mathcal{D}$-modules, we have $\mathcal{F}^2=(-1)^*$, i.e. pull-back by the map $(-1)\colon z\mapsto -z$.
\end{prop}

We will use the statement of the following lemma widely in this Section. Actually, it is a well-known fact, with several appearances in the literature, e.g. \cite{Lev}; \cite[Section~2]{Tur}, Theorem~I; \cite[Theorem~11.1]{Was}.

\begin{lemma}
    The irregular part at infinity $\widehat{\mathbb{M}}_{\widehat{\infty}}$ of the Fourier--Laplace transform of a holonomic $\mathcal{D}$-module $\mathbb{M}$ determined by a meromorphic connection of the form~\eqref{eq:locform} is independent of the residue $\Lambda_i$ and the holomorphic terms.
\end{lemma}

We note that Fang~\cite{Fang} expressed general formulas for the irregular type of local Fourier transformations using essentially a version of Legendre transformation, from which the results of this section could be derived too. 

\subsection{Cases with a pole of order $2$ at infinity}

The following lemma is valid in the cases $JKTVI$, $JKTV$, and $JKTIVa$, in which there is a logarithmic singularity at $0$, under some special assumption. Moreover, it accurately describes the structure of the module $\widehat{\mathbb{M}}_{\widehat{\infty}}$.

\begin{lemma}\label{prop:inf}
    In cases $JKTVI$, $JKTV$, and $JKTIVa$, if the residue of the connection matrix at $0$ is semisimple with one eigenvalue equal to $0$, then:
    \begin{itemize}
        \item[i)]  $\mathcal{F}^{\infty,\widehat{\infty}}(\mathbb{M})=0$.
        \item[ii)] $\operatorname{rk}(\mathcal{F}^{0,\widehat{\infty}}(\mathbb{M}))=2$, and $\operatorname{rk}(\widehat{\mathbb{M}}_{\widehat{\infty}})=2$. 
        \item[iii)] $\widehat{\mathbb{M}}$ has a regular singularity at $\widehat{\infty}$. 
    \end{itemize}
    \begin{proof}
        \begin{itemize}
            \item[i)] This follows immediately from Theorem~\ref{thm:LocalFourier}~\eqref{thm:LocalFourier3}, since in the cases $JKTVI$, $JKTV$, and $JKTIVa$, all slopes are $\lambda\leq1$ in $\mathbb{M}_{\infty}$ (see Table~\ref{tab:slopes}).
            \item[ii)] The proof is based on the technique of formal microlocalization, introduced in Subsection~\ref{sec:microloc}. More precisely, we are interested in the structure of the formal microlocalized module $\mathbb{M}_0^{\mu}$ (as a consequence of $i)$, we are not interested in $\mathbb{M}_{\infty}^{\mu}$ in this case).

            We can describe the singular module $\mathbb{M}_0$ as a rank three free module, divided by some defining relations, namely
            \begin{gather*}
                \mathbb{M}_0\cong\bigoplus_{i=1}^3\mathbb{C} \llbracket z \rrbracket \langle \partial_z \rangle q_i/(\partial_z,(\partial_z-\lambda_2z^{-1}),(\partial_z-\lambda_3z^{-1}))\cong \\
                \cong \mathbb{C} \llbracket z \rrbracket \langle \partial_z \rangle q_1/(\partial_z)\oplus \mathbb{C} \llbracket z \rrbracket \langle \partial_z \rangle q_2/(\partial_z-\lambda_2z^{-1})\oplus \mathbb{C} \llbracket z \rrbracket \langle \partial_z \rangle q_3/(\partial_z-\lambda_3z^{-1}),
            \end{gather*}
            where $q_i$ are suitable generators (namely, eigenvectors of the polar parts of the local forms of the connection at $0$), and the relations arise from the assumption on the residue at $0$. Notice that neither $\partial_z$ or $(\partial_z-\lambda_iz^{-1})$ ($i=2,3$) are units in $\mathbb{C} \llbracket z \rrbracket \langle \partial_z \rangle$, therefore, the direct summands are nonempty.

            On the other hand, in the course of the microlocalization, we consider the tensor product by 
            \[
            \mathcal{E}_0\cong\mathbb{C} \llbracket z \rrbracket \llbracket \hat{w} \rrbracket \langle \partial_z \rangle\cong \mathbb{C} \llbracket z \rrbracket \llbracket \hat{w} \rrbracket[\hat{w}^{-1}]
            \]
            (see \eqref{eq:micro}), and we have
            \[
            \mathcal{E}_0\otimes\mathbb{C} \llbracket z \rrbracket \langle \partial_z \rangle\cong\mathcal{E}_0.
            \]
            Therefore, the microlocalized (formal) module is
            \begin{equation}\label{eq:decomp1}
            \mathbb{M}_0^{\mu}\cong\mathcal{E}_0q_1/(\partial_z)\oplus \mathcal{E}_0 q_2/(\partial_z-\lambda_2z^{-1})\oplus \mathcal{E}_0 q_3/(\partial_z-\lambda_3z^{-1})
            \end{equation}
            Crucially now, as opposed to $\mathbb{C} \llbracket z \rrbracket \langle \partial_z \rangle$, in the ring $\mathcal{E}_0$ the element $\partial_z$ is a unit, because of the formal identification of $\partial_z^{-1}$ and $\hat{w}$. This means that the first direct summand $\mathcal{E}_0q_1/(\partial_z)=0$. 
            For $i=2,3$, the (formal) inverse of $(\partial_z-\lambda_iz^{-1})$ is, by repeated use of the noncommutative product rule of $\mathcal{E}_0$: 
            \begin{align*}
            (\partial_z-\lambda_iz^{-1})^{-1} & =\sum_{k=0}^{\infty}\lambda_i^k(\partial_z^{-1}z^{-1})^k\partial_z^{-1} \\
            & = \sum_{k=0}^{\infty}\lambda_i^k \left( \sum_{m = 0}^{\infty} m! z^{-1-m} \hat{w}^{m+1} \right)^k \hat{w} \\
            & = \hat{w} + \lambda_i z^{-1} \hat{w}^2 + \sum_{k=2}^{\infty} a_k z^{-k} \hat{w}^{k+1}, 
            \end{align*}
            for some $a_k\in \mathbb{C}$.  
            Since $\lambda_i\neq 0$, this series does not belong to $\mathcal{E}_0$, because it contains at least one negative power of $z$. 
            Therefore, the latter two direct summands are nonzero in \eqref{eq:decomp1}.
            
            It now follows that $\operatorname{rk}(\mathbb{M}_0^{\mu})=2$, and as a consequence of Theorem~\ref{thm:StationaryPhase}, we have $\operatorname{rk}(\widehat{\mathbb{M}}_{\widehat{\infty}})=2$.

            \item[iii)] In $\mathbb{M}_0$, for $i=2,3$ we have
            \[
            \partial_zq_i=\lambda_iz^{-1}q_i
            \]
            Applying the Fourier--Laplace anti-involution \eqref{eq:inv}, this becomes
            \[
            \hat{z}q_i=\lambda_i(-\partial_{\hat{z}})^{-1}q_i
            \]
            in $\widehat{\mathbb{M}}_{\widehat{\infty}}$ (or in $\mathbb{M}_0^{\mu}$). Rearranging this equation, we have:
            \[
            -\lambda_iq_i=\partial_{\hat{z}}\hat{z}q_i=-\hat{w}^2\partial_{\hat{w}}\hat{w}^{-1}q_i=-\hat{w}^2((-\hat{w})^{-2}q_i + \hat{w}^{-1}\partial_{\hat{w}}q_i)=q_i-\hat{w}\partial_{\hat{w}}q_i,
            \]
            where we used the identity $\partial_{\hat{z}}=-\hat{w}^2\partial_{\hat{w}}$. From this we get
            \[
            \partial_{\hat{w}}q_i=(\lambda_i+1)\hat{w}^{-1}q_i.
            \]
            This shows that we have a logarithmic singularity at $\widehat{\infty}$, where the polar part of the connection matrix is
            \[
            \begin{pmatrix}
                \lambda_2+1 & 0 \\
                0 & \lambda_3+1
            \end{pmatrix}\hat{w}^{-1}\otimes\operatorname{d}\! \hat{w}.
            \]
        \end{itemize}
    \end{proof}
\end{lemma}

\begin{rmrk}\label{rmrk:tensor}
    \begin{itemize}
        \item[i)] The condition of Lemma~\ref{prop:inf} on the residue at $0$ can be achieved from the general setup as well. Assume that $\Lambda_0$ is the residue of the connection matrix at $c=0$ (see \eqref{eq:locform}), and that the three eigenvalues of $\Lambda_0$ are $\lambda_1,\lambda_2,\lambda_3\in\mathbb{C}\setminus\{0\}$. Then consider the tensor product
            \[
            (E,\nabla)\otimes \left(L,\operatorname{d}-\lambda_1\frac{\operatorname{d}\! z}{z}\right),
            \]
            where $L$ is the trivial line bundle. The resulting connection satisfies the desired condition.
        \item[ii)] See also~\cite{Sz_Nahm} for a derivation of the statement in the  Dolbeault complex structure.
        \item[iii)] The eigenvalues of the residue of the logarithmic singularity at $\widehat{\infty}$ in part iii) of the lemma can simultaneously be shifted by an integer $k$ up to taking tensor product by $\mathcal{O}(-k\cdot \{ \widehat{\infty}\})$. 
    \end{itemize}
\end{rmrk}

In order to describe the module $\widehat{\mathbb{M}}$ in more detail in the cases $JKTVI$, $JKTV$, and $JKTIVa$, we analyze them case by case in the following.

\begin{assumption}\label{assumption}
    We make the following assumptions on the leading order terms of \eqref{connection1}-\eqref{connection3}.
    \begin{itemize}
        \item In \eqref{connection1}:
        \begin{equation}\label{leading1}
            -A_{-2}=\begin{pmatrix}
                0 & 0 & 0 \\
                0 & 1 & 0 \\
                0 & 0 & t
            \end{pmatrix}
        \end{equation}
        for some $t\neq0,1$ (i.e. $a_0=0$, $a_1=1$, $a_2=t$).
        \item In \eqref{connection2}:
        \begin{equation}\label{leading2}
            -A_{-2}=\begin{pmatrix}
                0 & 1 & 0 \\
                0 & 0 & 0 \\
                0 & 0 & t
            \end{pmatrix}
        \end{equation}
        for some $t\neq0$ (i.e. $a_0=0$, $a_1=t$).
        \item In \eqref{connection3}:
        \begin{equation}\label{leading3}
            -A_{-2}=\begin{pmatrix}
                0 & 1 & 0 \\
                0 & 0 & 1 \\
                0 & 0 & 0
            \end{pmatrix}
        \end{equation}
        (i.e. $a_0=0$).
    \end{itemize}
    We can assume these leading order terms, up to rescaling by a multiplicative constant, with respect to some basis, or similarly as in Remark~\ref{rmrk:tensor}~$i)$, by taking tensor product of $\mathbb{M}$ by a $\mathcal{D}_X$-module of rank $1$ with an irregular singularity of slope $1$ at $\infty$ and no other poles. 
    
\end{assumption}

\begin{prop}\label{prop:muinf}
     In cases $JKTVI$, $JKTV$, $JKTIVa$, we have $\mathbb{M}_{\infty}^{\mu}=0$.  
\end{prop}

\begin{rmrk}
    Although this statement follows immediately from the numerical data, see Theorem~\ref{thm:LocalFourier}~\eqref{thm:LocalFourier3} and Lemma~\ref{prop:inf} $i)$, we prove it here by the formal microlocalization technique, which could be considered as an alternative argument in these cases for Theorem~\ref{thm:LocalFourier}~\eqref{thm:LocalFourier3}. 
    The proof will be given under Assumption~\ref{assumption}, but the same proof goes through without it too. 
\end{rmrk}

\begin{proof}
    The steps of the proof are the same as in Lemma~\ref{prop:inf} $ii)$. In case $JKTVI$, we have
    \begin{gather*}
                \mathbb{M}_{\infty}
                \cong \mathbb{C} \llbracket w \rrbracket \langle \partial_w \rangle q_1/(\partial_w)\oplus \mathbb{C} \llbracket w \rrbracket \langle \partial_w \rangle q_2/(\partial_w-w^{-2})\oplus \mathbb{C} \llbracket w \rrbracket \langle \partial_w \rangle q_3/(\partial_w-tw^{-2}),
            \end{gather*}
    according to \eqref{leading1}. To obtain the microlocalized module $\mathbb{M}_{\infty}^{\mu}$, we take the tensor product with $\mathcal{E}_{\infty}$ from \eqref{eq:micro2}:
    \[
    \mathcal{E}_{\infty}\cong\mathbb{C}\llbracket z^{-1} \rrbracket\llbracket \hat{z}^{-1} \rrbracket [\hat{z}]\cong \mathbb{C}\llbracket w \rrbracket\llbracket \hat{w} \rrbracket [\hat{w}^{-1}],
    \]
    therefore,
    \[
    \mathcal{E}_{\infty}\otimes\mathbb{C} \llbracket w \rrbracket \langle \partial_w \rangle\cong\mathcal{E}_{\infty}[w^{-1}],
    \]
    because of the identity $\partial_w=-z^2\partial_z=-w^{-2}\hat{w}^{-1}$, that is, finitely many negative powers of $w$ can also appear. Now
    \begin{equation}\label{eq:microinfvi}
    \mathbb{M}_{\infty}^{\mu}
                \cong \mathcal{E}_{\infty}[w^{-1}] q_1/(\partial_w)\oplus \mathcal{E}_{\infty}[w^{-1}] q_2/(\partial_w-w^{-2})\oplus \mathcal{E}_{\infty}[w^{-1}] q_3/(\partial_w-tw^{-2}).
    \end{equation}
    Now $\partial_w^{-1}=-\hat{w}w^2$ formally, which is contained in $\mathcal{E}_{\infty}[w^{-1}]$, hence $\partial_w$ is a unit here, and the first direct summand is zero. For the latter two direct summands
    \[
    (\partial_w-tw^{-2})^{-1}=\sum_{k=0}^{\infty}t^k(\partial_w^{-1}w^{-2})^k\partial_w^{-1}=\sum_{k=0}^{\infty}t^k((-\hat{w}w^2)w^{-2})^k\partial_w^{-1}=\sum_{k=0}^{\infty}t^k(-\hat{w})^k\partial_w^{-1},
    \]
    which is in $\mathcal{E}_{\infty}[w^{-1}]$. Therefore, the latter two direct summands are also zero.
    The same argument can be repeated verbatim for $t$ replaced by $1$.

    In case $JKTV$ by \eqref{leading2}, we have
    \begin{gather}\label{eq:microinfv}
                \mathbb{M}_{\infty}
                \cong \left[\bigoplus_{i=1}^2\mathbb{C} \llbracket w \rrbracket \langle \partial_w \rangle q_i/\begin{pmatrix}
                    \partial_w & -1/w^2 \\
                    0 & \partial_w
                \end{pmatrix}\right]\oplus \mathbb{C} \llbracket w \rrbracket \langle \partial_w \rangle q_3/(\partial_w-tw^{-2}),
            \end{gather}
    meaning that the action on the first two direct summands is
    \[
    \begin{pmatrix}
                    \partial_w & -1/w^2 \\
                    0 & \partial_w
                \end{pmatrix} \begin{pmatrix}
                q_1 \\ q_2 \end{pmatrix}
                =\begin{pmatrix} \partial_wq_1-w^{-2}q_2\\ \partial_wq_2 \end{pmatrix}
    \]
    After the microlocalization, we get:
    \begin{gather*}
                \mathbb{M}_{\infty}^{\mu}
                \cong \left[\bigoplus_{i=1}^2\mathcal{E}_{\infty}[w^{-1}] q_i/\begin{pmatrix}
                    \partial_w & -1/w^2 \\
                    0 & \partial_w
                \end{pmatrix}\right]\oplus\mathcal{E}_{\infty}[w^{-1}] q_3/(\partial_w-tw^{-2}),
            \end{gather*}

    The third term is zero, as in the $JKTVI$ case. For the first two terms, the formal matrix acting on it is "invertible" in the sense that there exists $(\varepsilon_1,\varepsilon_2)^t\in\mathcal{E}_{\infty}[w^{-1}]\oplus\mathcal{E}_{\infty}[w^{-1}]$, such that
    \[
    \begin{pmatrix}
                    \partial_w & -1/w^2 \\
                    0 & \partial_w
                \end{pmatrix} \begin{pmatrix} \varepsilon_1\\ \varepsilon_2 \end{pmatrix} = \begin{pmatrix} 1\\ 1\end{pmatrix},
    \]
   equivalently
    \[
    \begin{pmatrix} \varepsilon_1\\ \varepsilon_2 \end{pmatrix} = \begin{pmatrix} \partial_w^{-1}(1+w^{-2}\partial_w^{-1}) \\ \partial_w^{-1} \end{pmatrix}.
    \]
     Here $(1,1)^t$ generates $\mathcal{E}_{\infty}[w^{-1}]\oplus\mathcal{E}_{\infty}[w^{-1}]$. Now $\partial_w^{-1}$ and $w^{-2}$ are in $\mathcal{E}_{\infty}[w^{-1}]$, therefore $(\varepsilon_1,\varepsilon_2)^t \in \mathcal{E}_{\infty}[w^{-1}]\oplus\mathcal{E}_{\infty}[w^{-1}]$, hence the first direct summand of $\mathbb{M}_{\infty}^{\mu}$ is also zero.

    In case $JKTIVa$, very similarly, using \eqref{leading3}:
    \begin{gather*}
                \mathbb{M}_{\infty}^{\mu}
                \cong \bigoplus_{i=1}^3\mathcal{E}_{\infty}[w^{-1}] q_i/\begin{pmatrix}
                    \partial_w & -1/w^2 & 0 \\
                    0 & \partial_w & -1/w^2 \\
                    0 & 0 & \partial_w
                \end{pmatrix}=0,
            \end{gather*}
    where the last identity is true because for the vector
    \[
    \begin{pmatrix} \varepsilon_1 \\ \varepsilon_2 \\ \varepsilon_3 \end{pmatrix} = \begin{pmatrix} \partial_w^{-1}(1+w^{-2}\partial_w^{-1}+(w^{-2}\partial_w^{-1})^2) \\ \partial_w^{-1}(1+w^{-2}\partial_w^{-1})\\ \partial_w^{-1} \end{pmatrix} \in\mathcal{E}_{\infty}[w^{-1}]^{\oplus3},
    \]
    the action of the connection matrix on $(\varepsilon_1,\varepsilon_2,\varepsilon_3)^t$ takes it to $(1,1,1)^t$, which generates the whole ring.
    
\end{proof}

\subsubsection{The $JKTVI$ case}\label{sec:dR_P6}

Let us assume that $\mathbb{M}$ has 
\begin{itemize}
    \item a regular singular point at $0$ with non-resonant eigenvalues of the residue, as assumed in Lemma~\ref{prop:inf}, and 
    \item an irregular singularity at $z= \infty$ where the irregular type has a pole of order $1$ whose (leading) degree $-1$ matrix $-A_{-2}$ has three different eigenvalues, and one of them is $0$, see \eqref{leading1}. 
\end{itemize}

As a consequence of Lemma~\ref{prop:inf}, $\widehat{\mathbb{M}}$ has a regular singularity at $\widehat{\infty}$ and $\operatorname{rk}(\widehat{\mathbb{M}}_{\widehat{\infty}})=2$. Then, the only irregular types that appear in the HLT-decomposition are $q = w^{-1}$ and $q = t w^{-1}$. Theorem~\ref{thm:StationaryPhase} implies that $\hat{0},-\hat{1},-\hat{t}$ are regular singular points of $\widehat{\mathbb{M}}$, and the residue matrices at these points have rank $1$. Also, Theorem~\ref{thm:StationaryPhase} shows that no more singularity appears. In summary, applying Fourier--Laplace transformation, we get $\operatorname{rk}(\widehat{\mathbb{M}})=2$ with regular singularities at $\hat{0},-\hat{1},-\hat{t},\widehat{\infty}$. This is the JMU representation of the Painlev\'e VI equation.

Actually, all the slopes at $\infty$ are $\leq 1$, with one slope being $<1$, because of the zero eigenvalue of $-A_{-2}$. This assumption is necessary because otherwise all the slopes at $\infty$ would be $\geq 1$, and Theorem~\ref{thm:LocalFourier}~\ref{thm:LocalFourier2} would imply $\mathcal{F}^{\infty,\hat{0}}(\mathbb{M}) = 0$.

The arrangement of the regular singularities at $\hat{0},-\hat{1},-\hat{t}$ can also be obtained from the microlocalized module $\mathbb{M}_{\infty}^{\mu}$. In the third direct summand of \eqref{eq:microinfvi} it is $\partial_wq_3=tw^{-2}q_3$. Multiplying by $w^2$:
    \[
    tq_3=w^2\partial_wq_3=-\partial_zq_3=-\hat{z}q_3.
    \]
    Now multiplying by $\partial_{\hat{z}}$:
    \[
    t\partial_{\hat{z}}q_3=-q_3-\hat{z}\partial_{\hat{z}}q_3,
    \]
    that is $\partial_{\hat{z}}q_3=-(\hat{z}+t)^{-1}q_3$. We see that there is a rank one regular singularity at $\hat{z}=-\hat{t}$. The argument is the same for $q_1,q_2$ and $\hat{z}=\hat{0},-\hat{1}$.

\subsubsection{The $JKTV$ case}\label{sec:dR_P5}

In this case, $\mathbb{M}$ has 
\begin{itemize}
    \item a regular singular point at $0$ with non-resonant eigenvalues of the residue, with one of them being zero, as assumed in Lemma~\ref{prop:inf}, and 
    \item an irregular singularity at $z= \infty$ where the irregular type has a pole of order $1$ whose (leading) degree $-1$ matrix $-A_{-2}$ is regular with two different eigenvalues: $0$ with multiplicity $2$ and another eigenvalue (say, $t$) with multiplicity $1$, see \eqref{leading2}. 
\end{itemize}

Again, Lemma~\ref{prop:inf} shows that $\widehat{\mathbb{M}}$ has a regular singularity at $\widehat{\infty}$, and $\operatorname{rk}(\widehat{\mathbb{M}}_{\widehat{\infty}})=2$. The only irregular type appearing in the HLT-decomposition is $q = t w^{-1}$. Theorem~\ref{thm:StationaryPhase} implies that $\hat{t}$ (we can choose $t=1$) is a regular singular point of $\widehat{\mathbb{M}}$, and the residue matrices at these points have rank $1$. Also, Theorem~\ref{thm:StationaryPhase} shows that $\hat{0}$ is the only possible singularity remaining. Indeed, Theorem~\ref{thm:LocalFourier}~\eqref{thm:LocalFourier2} shows that $\hat{z}=\hat{0}$ is an irregular singular point of Swan conductor $1$. In order to analyze $\widehat{\mathbb{M}}_{\hat{0}}=\mathcal{F}^{(\infty,\hat{0})}(\mathbb{M}_{\infty})$, consider $\mathbb{M}_{\infty}^{\mu}$.

As in Section~\ref{sec:dR_P6}, $\partial_wq_3=tw^{-1}q_3$ from \eqref{eq:microinfv} implies the regular singularity of rank one at $\hat{z}=\hat{t}$. Consider the other two relations from \eqref{eq:microinfv} under the same computation as in Section~\ref{sec:dR_P6}:
\begin{equation*}
    \begin{cases}
        \partial_wq_1=w^{-2}q_2 \\
        \partial_wq_2=0
    \end{cases}\hspace{0.5cm}\leftrightarrow\hspace{0.5cm} \begin{cases}
        -q_1-\hat{z}\partial_{\hat{z}}q_1=\partial_{\hat{z}}q_2 \\
        \partial_{\hat{z}}q_2=-\hat{z}^{-1}q_2
    \end{cases}
\end{equation*}

Substituting the second equation into the first one, we receive
\[
\partial_{\hat{z}}q_1=\hat{z}^{-2}q_2-\hat{z}^{-1}q_1
\]

We see that $\widehat{\mathbb{M}}_{\hat{0}}$ has the connection matrix with polar part
\[
    \left[\begin{pmatrix}
        0 & 1 \\
        0 & 0
    \end{pmatrix}
    \hat{z}^{-2} + \begin{pmatrix}
        -1 & 0 \\
        0 & -1
    \end{pmatrix} \hat{z}^{-1}\right]\otimes\operatorname{d}\hat{z}
\]
All in all, application of Fourier--Laplace transformation results $\operatorname{rk}(\widehat{\mathbb{M}})=2$, in the JMU representation of the degenerate Painlev\'e V equation. 

\subsubsection{The $JKTIVa$ case}

Here, the structure of $\mathbb{M}$:
\begin{itemize}
    \item at $z=0$ is the same as in Section~\ref{sec:dR_P6}-\ref{sec:dR_P5}, and
    \item at $z= \infty$ the irregular type has a nilpotent degree $-1$ leading term in $Q$, see \eqref{leading3}.
\end{itemize} 

Again, $\widehat{\mathbb{M}}$ is of rank $2$, based on Theorem~\ref{thm:StationaryPhase}, we can describe the set of singular points of $\widehat{\mathbb{M}}$. There is no more singularity, but $\hat{0}$ and $\widehat{\infty}$ (all slopes of $\mathbb{M}_{\infty}$ are $<1$, see Table~\ref{tab:slopes}). Now $\widehat{\mathbb{M}}_{\widehat{\infty}}$ is described by Lemma~\ref{prop:inf}: it has a regular singularity at $\widehat{\infty}$.

By virtue of Theorem~\ref{thm:LocalFourier}~\eqref{thm:LocalFourier2}, we have 
\[
 \operatorname{Sw}(\widehat{\mathbb{M}}_{\hat{0}}) = \frac 23 \cdot 3 = 2, \quad \operatorname{rk}(\widehat{\mathbb{M}}_{\hat{0}}) = 3 - 2 = 1,
\]
which shows that in $\hat{z}=\hat{0}$ we have a rank one, order three pole. 

Taking everything together, we have $\operatorname{rk}(\widehat{\mathbb{M}})=2$ and the connection matrix of $\widehat{\mathbb{M}}$ has a regular singularity at $\widehat{\infty}$ and a pole of order $3$ at $\hat{0}$, but only one eigenvalue of the irregular part is nonzero. This is therefore a special case of Painlev\'e IV.

\subsection{Cases with a pole of order $3$ at infinity}

\begin{assumption}
    We make the following assumptions on the irregular terms of \eqref{connection4}-\eqref{connection6}.
    \begin{itemize}
        \item In \eqref{connection4}:
        \begin{equation}\label{leading4}
            -2A_{-3}=\begin{pmatrix}
                0 & 0 & 0 \\
                0 & t_1 & 0 \\
                0 & 0 & t_2
            \end{pmatrix},\hspace{0.5cm} -A_{-2}=
            \begin{pmatrix}
                0 & 0 & 0 \\
                0 & b_1 & 0 \\
                0 & 0 & b_2
            \end{pmatrix}
        \end{equation}
        for some different $t_1,t_2\neq0$, that is $a_0=0$, $a_1=t_1$, $a_2=t_2$ $b_0=0$.
        \item In \eqref{connection5}:
        \begin{equation}\label{leading5}
            -2A_{-3} = \begin{pmatrix}
        0 & 1 & 0 \\
        0 & 0 & 0 \\
        0 & 0 & t 
    \end{pmatrix}, 
    \quad
    -A_{-2} = \begin{pmatrix}
        0 & 0 & 0 \\
        b & 0 & 0 \\
        0 & 0 & 0
    \end{pmatrix}
        \end{equation}
        for some $t,b\neq0$ (i.e. $a_0=0$, $a_1=t$, $b_1,b_2=0$, $b_0=b$).
        \item In \eqref{connection6}:
        \begin{equation}\label{leading6}
            -2A_{-3} = \begin{pmatrix}
        0 & 1 & 0 \\
        0 & 0 & 1 \\
        0 & 0 & 0
    \end{pmatrix}, 
    \quad
    -A_{-2} = \begin{pmatrix}
        0 & 0 & 0 \\
        0 & 0 & 0 \\
        b & 0 & 0
    \end{pmatrix}
        \end{equation}
        for some $b\neq 0$ (i.e. $a_0=0$, $b_1,b_2=0$, $b_0=b$).
    \end{itemize}
    The same argument as in Assumption~\ref{assumption} shows that these are natural choices for the parameters. 
    The conditions $b\neq 0$ are needed to ensure that the slopes are indeed $3/2$ and $5/3$ and not smaller. 
\end{assumption}

\begin{prop}\label{prop:1}
    In cases $JKTIVb$, $JKTII$ and $JKTI$: $\operatorname{rk}(\mathbb{M}_{\infty}^{\mu})=2$.
\end{prop}

\begin{proof}
The microlocalized modules $\mathbb{M}_{\infty}^{\mu}$ are established in the same way as in Proposition~\ref{prop:muinf}. Therefore, in case $JKTIVb$, we have
\begin{equation}\label{eq:microinfivb}
    \mathbb{M}_{\infty}^{\mu}
                \cong \mathcal{E}_{\infty}[w^{-1}] q_1/(\partial_w)\oplus \mathcal{E}_{\infty}[w^{-1}] q_2/(\partial_w-t_1w^{-3}-b_1w^{-2})\oplus \mathcal{E}_{\infty}[w^{-1}] q_3/(\partial_w-t_2w^{-3}-b_2w^{-2}),
    \end{equation}
    where the relations come from the local form of the leading order term \eqref{leading4}. Recall that
    \begin{equation*}
    \mathcal{E}_{\infty}[w^{-1}]\cong\mathbb{C}\llbracket w \rrbracket\llbracket \hat{w} \rrbracket [\hat{w}^{-1}][w^{-1}],
    \end{equation*}
    therefore, $\partial_w^{-1}=-\hat{w}w^2\in\mathcal{E}_{\infty}[w^{-1}]$, and the first direct summand of \eqref{eq:microinfivb} is zero. On the other hand (for $i=1,2$),
    \begin{gather*}
    (\partial_w-t_iw^{-3}-b_iw^{-2})^{-1}=\sum_{k=1}^{\infty}(\partial_w^{-1}(t_iw^{-3}+b_iw^{-2}))^k\partial_w^{-1}= \\ \sum_{k=0}^{\infty}((-\hat{w}w^2)(t_iw^{-3}+b_iw^{-2}))^k\partial_w^{-1}=
    \sum_{k=0}^{\infty}(-t_i\hat{w}w^{-1}-b_i\hat{w})^{k}\partial_w^{-1}.
    \end{gather*}
    The homogeneous terms of degree $-m$ in $w$ all come from the series 
    \[
    \sum_{k=m}^{\infty} \binom km (-t_i)^m (-b_i)^{k-m} (\hat{w}w^{-1})^m \hat{w}^{k-m} \partial_w^{-1}.
    \]
    Using the commutator relation $[\hat{w},w] = - \hat{w}^2$,  we can move all factors $w^{-1}$ to the left. 
    To move the first occurrence of $w^{-1}$, we need to commute it with just one $\hat{w}$, for the second one we need to commute it with two $\hat{w}$'s, and so on, until its $m$'th occurrence that requires to be commuted with $m$ $\hat{w}$'s. 
    All in all, we need $m(m+1)/2$ commutators. 
    This implies that coefficient of $w^{-m}$ in the above formal power series is 
    \[
        (-1)^{\frac{m(m+1)}2} \hat{w}^{m(m+1)} \sum_{k=m}^{\infty} \binom km (-t_i)^m (-b_i)^{k-m} \hat{w}^k \partial_w^{-1}. 
    \]
    By assumption, $t_i\neq 0 \neq b_i$. 
    Therefore, infinitely many negative powers of $w$ appear with nonzero coefficient, hence it cannot be contained in $\mathcal{E}_{\infty}[w^{-1}]$. The key difference between this and Proposition~\ref{prop:muinf} is that here $w$ has power $-3$, not $-2$, which is not canceled in the formula. This shows that the latter two direct summands in \eqref{eq:microinfivb} are nonzero, and $\mathbb{M}_{\infty}^{\mu}$ is of rank 2.

    In the case $JKTII$, by \eqref{leading5}, we have
    \begin{gather}\label{eq:microinfii}
                \mathbb{M}_{\infty}^{\mu}
                \cong \left[\bigoplus_{i=1}^2\mathcal{E}_{\infty}[w^{-1}] q_i/\begin{pmatrix}
                    \partial_w & -1/w^3 \\
                    -b/w^2 & \partial_w
                \end{pmatrix}\right]\oplus \mathcal{E}_{\infty}[w^{-1}] q_3/(\partial_w-tw^{-3}),
            \end{gather}
    As in the previous case, the second direct summand is nonzero and has rank one. For the first direct summand, we are looking for an $(\varepsilon_1,\varepsilon_2)^t$ such that 
    \[
    \begin{pmatrix}
                    \partial_w & -1/w^3 \\
                    -b/w^2 & \partial_w
                \end{pmatrix} \begin{pmatrix}
                    \varepsilon_1 \\ \varepsilon_2
                \end{pmatrix}= \begin{pmatrix} 1 \\ 1 \end{pmatrix}.
    \]
    After computation, we get
    \[
    \begin{pmatrix}
                    \varepsilon_1 \\ \varepsilon_2
                \end{pmatrix}
                =\begin{pmatrix} (\partial_w-bw^{-3}\partial_w^{-1}w^{-2})^{-1}(1+w^{-2}) \\ bw^{-2}(\partial_w-bw^{-3}\partial_w^{-1}w^{-2})^{-1}(1+w^{-2}))\end{pmatrix}
    \]
    which is not in $\mathcal{E}_{\infty}[w^{-1}]\oplus\mathcal{E}_{\infty}[w^{-1}]$, because
    \[
    (\partial_w-bw^{-3}\partial_w^{-1}w^{-2})^{-1}=\sum_{k=1}^{\infty}b^k(\partial_w^{-1}w^{-3}\partial_w^{-1}w^{-2})^k\partial_w^{-1}=\sum_{k=1}^{\infty}b^k(\hat{w}w^{-1}\hat{w})^k\partial_w^{-1},
    \]
    where infinitely many negative powers of $w$ appear, just as in case IVb above. 
    This means that the rank of the first direct summand in \eqref{eq:microinfii} is nonzero, the question is whether it is one or two. Consider the relations and their formal transformations in the microlocalized module, via the Fourier--Laplace anti-involution \eqref{eq:inv}:
    \begin{equation}\label{comp1}
        \begin{cases}
            \partial_wq_1=w^{-3}q_2 \\
            \partial_wq_2=bw^{-2}q_1
        \end{cases}\quad\leftrightarrow \begin{cases}
            -\hat{w}^{-1}q_1=\hat{w}^2\partial_{\hat{w}}q_2 \\
            -\hat{w}^{-1}q_2=bq_1
        \end{cases}
    \end{equation}
    Here, the second equation, after the transformation, establishes a $\mathbb{C}[\hat{w}^{-1}]$-linear relation between the generators $q_1$ and $q_2$. Therefore, the rank of the first direct summand of \eqref{eq:microinfii} drops by one, and the total rank of $\mathbb{M}_{\infty}^{\mu}$ turns out to be 2.

    Moreover, 
    \begin{equation}\label{comp2}
        \partial_{\hat{w}}q_2=-\hat{w}^{-3}q_1=b^{-1}\hat{w}^{-4}q_2
    \end{equation}

    In the case $JKTI$, the process is very similar. Consider the following microlocalized module (based on \eqref{leading6}):
    \begin{gather}\label{eq:microinfi}
                \mathbb{M}_{\infty}^{\mu}
                \cong \bigoplus_{i=1}^3\mathcal{E}_{\infty}[w^{-1}] q_i/\begin{pmatrix}
                    \partial_w & -1/w^3 & 0 \\
                    0 & \partial_w & -1/w^3 \\
                    -b/w^{2} & 0 & \partial_w
                \end{pmatrix},
            \end{gather}
    After similar tedious computation as in the previous case, we see that the $(\varepsilon_1,\varepsilon_2,\varepsilon_3)^t$ which is taken to $(1,1,1)^t$ by the action of the above matrix, is not contained in $\mathcal{E}_{\infty}[w^{-1}]^{\oplus 3}$. Therefore, the rank of \eqref{eq:microinfi} is nonzero. On the other hand, consider the formal transformations of the relations in the microlocalized module:
    \begin{equation}\label{comp4}
        \begin{cases}
            \partial_wq_1=w^{-3}q_2 \\
            \partial_wq_2=w^{-3}q_3 \\
            \partial_wq_3=bw^{-2}q_1
        \end{cases}\quad\leftrightarrow \begin{cases}
            -\hat{w}^{-1}q_1=\hat{w}^2\partial_{\hat{w}}q_2 \\
            -\hat{w}^{-1}q_2=\hat{w}^2\partial_{\hat{w}}q_3 \\
            -\hat{w}^{-1}q_3=bq_1
        \end{cases}
    \end{equation}
    The third equation shows again a $\mathbb{C}[\hat{w}^{-1}]$-linear relation between two generators of $\mathbb{M}_{\infty}^{\mu}$, hence the rank drops by one and becomes two. Furthermore,
    \begin{equation}\label{comp5}
    \begin{aligned}
        \partial_{\hat{w}}q_2 & =-\hat{w}^{-3}q_1=b^{-1}\hat{w}^{-4}q_3 \\
        \partial_{\hat{w}}q_3 & =-\hat{w}^{-3}q_2
    \end{aligned}
    \end{equation}

\end{proof}

\begin{prop}\label{prop:2}
    In case $JKTIVb$ there are two singular points of $\widehat{\mathbb{M}}$: $\hat{0}$ and $\widehat{\infty}$. In cases $JKTII$ and $JKTI$ the point $\widehat{\infty}$ is the only singularity of $\widehat{\mathbb{M}}$.
\end{prop}

\begin{proof}
    In all three cases $\infty$ is the only singularity of $\mathbb{M}$. In the latter two cases, there is no summand of the form
    \[
    Q_j=q_jw^{-1}+o(w^{-1})
    \]
    in the HLT-decomposition of $\mathbb{M}_{\infty}$. Therefore in the $JKTI$, $JKTII$ cases Theorem~\ref{thm:StationaryPhase} excludes all possible singularities in $\widehat{\mathbb{M}}$, but $\hat{0}$ and $\widehat{\infty}$. On the other hand, Theorem~\ref{thm:LocalFourier}~\ref{thm:LocalFourier2} implies 
    \[
    \mathcal{F}^{\infty,\hat{0}}(\mathbb{M}) = 0; 
    \]
    because $\mathbb{M}_{\infty}$ only has slope $\lambda\geq 1$ summands. Hence $\hat{0}$ is not a singularity of $\widehat{\mathbb{M}}$. But $\widehat{\infty}$ is a singularity, as we have seen in Proposition~\ref{prop:1}, $\mathbb{M}_{\infty}^{\mu}$ is nonzero, which is isomorphic to $\widehat{\mathbb{M}}_{\widehat{\infty}}$ via Theorem~\ref{thm:Fsp}.

    The difference in case $JKTIVb$ is that the choice for one eigenvalue of $-2A_{-3}$ and $-A_{-2}$ is zero (which belongs to the same eigenspace) in \eqref{leading4}. This means that there is a slope $\lambda<1$ in $\mathbb{M}_{\infty}$, and $\mathcal{F}^{\infty,\hat{0}}(\mathbb{M}) \neq 0$. Then Theorem~\ref{thm:StationaryPhase} implies that $\hat{0}\in\widehat{\mathbb{M}}$ is a regular singularity with a residue matrix of rank 1. The same way as in the other two cases $\widehat{\infty}\in\widehat{\mathbb{M}}$ is also a singularity.
\end{proof}

\subsubsection{The $JKTIVb$ case}

The conditions on $\mathbb{M}$ are: 
\begin{itemize}
    \item a unique singularity at $z= \infty$ where the irregular type has a pole of order $2$ whose (leading) degree $-2$ matrix $-2A_{-3}$ is regular semisimple with one eigenvalue equal to $0$ and the next matrix $-A_{-2}$ also has a $0$ eigenvalue with the same eigenspace as $-2A_{-3}$ (see \eqref{leading4}). 
\end{itemize}

Then, as a consequence of Propositions~\ref{prop:1},\ref{prop:2} $\widehat{\mathbb{M}}$ has a regular singularity at $\hat{z} = \hat{0}$ with residue of rank $1$, and an irregular singularity at $\widehat{\infty}$. As we have seen in Proposition~\ref{prop:1}, $\operatorname{rk}(\mathbb{M}_{\infty}^{\mu})=2$. Similarly as in \eqref{comp1} and \eqref{comp4}, the formal transformation of the defining relations in the microlocalized module is:
\begin{equation}\label{eq:ivb}
\partial_wq_3=t_2w^{-3}q_3+b_2w^{-2}q_3 \quad \leftrightarrow \quad \partial_{\hat{w}}q_3=-t^{-1}\hat{w}^{-3}q_3-b_2t_2^{-1}\hat{w}^{-2},
\end{equation}
and similarly for $q_2$ and $t_1$. This implies by virtue of Theorem~\ref{thm:Fsp} that the polar part of the connection matrix belonging to $\widehat{\mathbb{M}}_{\widehat{\infty}}$ has irregular part

\[
\left[\begin{pmatrix}
    -t_1^{-1} & 0 \\
    0 & -t_2^{-1}
\end{pmatrix}\hat{w}^{-3}+\begin{pmatrix}
    -b_1t_1^{-1} & 0 \\
    0 & -b_2t_2^{-1}
\end{pmatrix}\hat{w}^{-2}\right]\otimes\operatorname{d}\! \hat{w}.
\]

The result of the Fourier--Laplace transformation is a rank two singularity at $\widehat{\infty}$ which is a pole of order three and a rank one singularity at $\hat{0}$ which is a pole of order one. All together, this is the JMU representation of the general case of Painlev\'e IV.

\begin{rmrk}
    The numerical data from Theorem~\ref{thm:LocalFourier} satisfy
    \[
        \operatorname{Sw}(\widehat{\mathbb{M}}_{\widehat{\infty}}) = \operatorname{Sw}(\mathbb{M}_{\infty}) = 4, \quad \operatorname{rk}(\widehat{\mathbb{M}}_{\widehat{\infty}}) = 4 - 2 = 2.
    \]
    We see that the Fourier--Laplace transformation just decreases the rank of the irregular singularity by $1$, although the exact irregular type cannot be concluded from the numerical data.
    
    If we choose generic eigenvalues for $-2A_{-3}$, then the Laplace transform will map it to a similar object, i.e. the set of generic $\operatorname{rk}=3, \operatorname{Sw}=6$ connections with $\infty$ as only singularity, is self-dual. 
\end{rmrk}

\subsubsection{The $JKTII$ case}

Here, the only (irregular) singularity of $\mathbb{M}$ being placed at $\infty$, and we choose $-2A_{-3}$ and $-A_{-2}$ as described in \eqref{leading5}. By Proposition~\ref{prop:2} $\widehat{\mathbb{M}}$ has only one singularity, which is placed at $\widehat{\infty}$. By Proposition~\ref{prop:2} and Theorem~\ref{thm:Fsp}, $\widehat{\mathbb{M}}_{\widehat{\infty}}$ is of rank two. The numerical data from Theorem~\ref{thm:LocalFourier}~\eqref{thm:LocalFourier3} shows 
\[
       \operatorname{Sw}(\widehat{\mathbb{M}}_{\widehat{\infty}}) = \operatorname{Sw}(\mathbb{M}_{\infty}) = 2 \cdot \frac 32 + 2 = 5, \quad \operatorname{rk}(\widehat{\mathbb{M}}_{\widehat{\infty}}) = 5 - 3 = 2.
\]
Thus, the singularity is a pole of order four. Computations \eqref{comp1}-\eqref{comp2} in $\mathbb{M}_{\infty}^{\mu}$ also verify this. Moreover, together with \eqref{eq:ivb} (which is the same here for $q_3$), these computations show that the connection matrix belonging to $\widehat{\mathbb{M}}_{\widehat{\infty}}$, has irregular part in the basis $\{q_2,q_3\}$:
\[
    \left[\begin{pmatrix}
        b^{-1} & 0 \\
        0 & 0
    \end{pmatrix}\hat{w}^{-4}+ \begin{pmatrix}
    0 & 0 \\
    0 & - t^{-1}
    \end{pmatrix}\hat{w}^{-3}+O(\hat{w}^{-2})\right]\otimes\operatorname{d}\hat{w}. 
\]
Formally, $\widehat{\mathbb{M}}_{\widehat{\infty}}$ decomposes into a sum of two rank $1$ modules of slopes $3$ and $2$, respectively. The only singularity is of rank 2, placed at $\widehat{\infty}$, which is a pole of order four with the above irregular part. This provides a special case of the JMU representation of Painlev\'e II, namely, one eigenvalue of the leading order term vanishes.

\subsubsection{The $JKTI$ case}

Finally, in this case, we have a single singularity at $\infty$, which is a pole of order three, with the coefficients of the irregular type given by \eqref{leading6}. As in the previous case, by Proposition~\ref{prop:2} $\widehat{\mathbb{M}}$ has only one singularity, which is placed at $\widehat{\infty}$. By Proposition~\ref{prop:2} and Theorem~\ref{thm:Fsp} $\widehat{\mathbb{M}}_{\widehat{\infty}}$ is of rank two. The numerical data from Theorem~\ref{thm:LocalFourier}~\eqref{thm:LocalFourier3} shows
\[
       \operatorname{Sw}(\widehat{\mathbb{M}}_{\widehat{\infty}}) = \operatorname{Sw}(\mathbb{M}_{\infty}) = 3 \cdot \frac 53 = 5, \quad \operatorname{rk}(\widehat{\mathbb{M}}_{\widehat{\infty}}) = 5 - 3 = 2. 
\]
This means that the singularity at $\widehat{\infty}$ is of rank two, and is a pole of order four. Again, \eqref{comp4}-\eqref{comp5} show the same, and according to \eqref{comp5}, the irregular type of the connection matrix in basis $\{q_2,q_3\}$ of $\widehat{\mathbb{M}}_{\widehat{\infty}}$ is 
\[
    \left[\begin{pmatrix}
        0 & b^{-1} \\
        0 & 0
    \end{pmatrix}\hat{w}^{-4}+ \begin{pmatrix}
    0 & 0 \\
    -1 & 0
    \end{pmatrix}\hat{w}^{-3}+O(\hat{w}^{-2})\right]\otimes\operatorname{d}\hat{w}. 
\]
This gives the JMU representation of Painlev\'e I.

\bigskip
\textbf{Funding}
\bigskip
\\ The project supported by the Doctoral Excellence Fellowship Programme (DCEP) is funded by the National Research Development and Innovation Fund of the Ministry of Culture and Innovation and the Budapest University of Technology and Economics, under a grant agreement with the National Research, Development and Innovation Office. During the preparation of this manuscript, the authors were supported by the grant K146401 of the National Research, Development and Innovation Office.
The second author was also supported by the grant KKP144148 of the National Research, Development and Innovation Office.


\begin{thebibliography}{999999999}

\bibitem{AM} {\sc M.~Alameddine, O.~Marchal:} Explicit Hamiltonian representations of meromorphic connections and duality from different perspectives: a case study, \texttt{arXiv:2406.19187}

\bibitem{BJL} {\sc W.~Balser, W.~B.~Jurkat, D.~A.~Lutz:} On the reduction of connection problems for differential
equations with an irregular singularity to ones with only regular singularities, I. \textit{SIAM J. Math. Anal.} {\bf{12}}, 5 (1981), \rm 691-721. 

\bibitem{BB} {\sc O.~Biquard, P.~Boalch:} Wild non-abelian {H}odge theory on curves, \textit{Compos. Math.} {\bf{140}}, 1 (2004), \rm 179--204.

\bibitem{BE} {\sc S.~Bloch, H.~Esnault:} Local Fourier Transforms and Rigidity for D-Modules, \textit{Asian Journal of Mathematics} {\bf{8}} (2004) 

\bibitem{Boa5} {\sc P.~Boalch:} Simply-laced isomonodromy systems, \textit{Publ. Math. IHES} {\bf{116}}, 1 (2012), \rm 1-68. 

\bibitem{Del} {\sc P.~Deligne:} \'Equations diff\'erentielles \`a points singuliers r\'eguliers, Lecture Notes in Mathematics {\bf 163} (1970), Springer

\bibitem{Dou} {\sc J.~Dou\c{c}ot:} Basic representations of genus zero nonabelian Hodge spaces, \texttt{arXiv:2409.12864}

\bibitem{ESz2} {\sc M.~Eper, Sz.~Szab\'o:} Rank three representations of Painlev\'e systems: I. Wild character varieties, \texttt{arxiv:2505.21186}

\bibitem{ESz4} {\sc M.~Eper, Sz.~Szab\'o:} Rank three representations of Painlev\'e systems: III. Dolbeault structure, spectral correspondence, \texttt{arXiv:2509.00418}

\bibitem{Fang} {\sc J-X. Fang:} Calculation of local Fourier transforms for formal
connections, \textit{Science in China Series A: Mathematics} {\bf 52}, 10 (2009), \rm 2195--2206.

\bibitem{GL} {\sc R.~Garc\'{\i}a-L\'opez:} Microlocalization and stationary phase, \textit{Asian Journal of Mathematics} {\bf{8}} (2004) 

\bibitem{Har} {\sc J.~Harnad:} Dual Isomonodromic Deformations and Moment Maps to Loop Algebras, \textit{Commun. Math. Phys.} {\bf{166}} (1994) \rm 337-365.

\bibitem{Huk} {\sc M.~Hukuhara:} Sur les points singuliers des {\'e}quations diff{\'e}rentielles lin{\'e}aires, III \textit{Memoirs of the Faculty of Science, Kyushu University. Series A, Mathematics}, {\bf{2}}, 2 (1942) \rm 125-137.

\bibitem{JKTI} {\sc N.~Joshi, A.~Kitaev, P.~Treharne:} On the Linearization of the Painlev\'e III-VI Equations and Reductions of the Three-Wave Resonant System, \textit{Journal of Mathematical Physics} {\bf{48}}, 10 (2007) 

\bibitem{JKTII} {\sc N.~Joshi, A.~Kitaev, P.~Treharne:} On the Linearization of the First and Second Painlev\'e Equations, \textit{Journal of Physics A: Mathematical and Theoretical} {\bf{42}}, 5 (2009)


\bibitem{Lau} {\sc G.~Laumon:} Transformation de Fourier, constantes d'\'equations fonctionnelles et conjecture de Weil, \textit{Publ. Math. IHES} (1987)

\bibitem{Lev}  {\sc  A.~H.~M.~Levelt:}  Jordan decomposition for a class of singular differential operators, \textit{Arkiv f{\"o}r matematik} 
  {\bf{13}} (1975), \rm 1--27

\bibitem{Malgr} {\sc B.~Malgrange:} \'Equations diff\'erentielles \`a coefficients polynomiaux, \textit{Progress in Mathematics} {\bf{96}}, Birkh\"auser (1991)

\bibitem{Maz} {\sc M.~Mazzocco:} Painlev\'e sixth equation as isomonodromic deformations equation of an irregular system, \textit{CRM Proc. Lec. Notes} {\bf{32}}, Amer. Math. Soc., Providence, RI, (2002)  219--238.	


\bibitem{Sab2} {\sc C.~Sabbah:} Harmonic metrics and connections with irregular singularities, \textit{Ann. Inst. Fourier (Grenoble)} {\bf{49}}, (1999) \rm 1265--1291.

\bibitem{Sab} {\sc C.~Sabbah:} Fourier-Laplace transform of a variation of polarized complex Hodge structure, \textit{Journal F\"ur Die Reine Und Angewandte Mathematik} (2008)

\bibitem{Sab_Frobenius} {\sc C.~Sabbah:} Isomonodromic deformations and Frobenius manifolds: an introduction, \textit{Universitext} \textit{Springer}  (2008) 

\bibitem{Sz_Nahm} {\sc Sz.~Szab\'o:} Nahm transform for integrable connections 
on the Riemann sphere, \textit{M\'emoires de la Soci\'et\'e Math\'ematique de France} {\bf 110}, (2007)

\bibitem{Sz_Laplace} {\sc Sz.~Szab\'o:} Nahm transform and parabolic minimal Laplace transform, \textit{Journal of Geometry and Physics} {\bf{62}}, 11  (2012), \rm 2241--2258.

\bibitem{Sz_Plancherel} {\sc Sz.~Szab\'o:} The Plancherel theorem for Fourier--Laplace--Nahm transform for connections on the projective line,  \textit{Commun. Math. Phys.} {\bf 338}, 2 (2015), \rm 753--769. 

\bibitem{Tur} {\sc H.~L.~Turrittin:} Convergent solutions of ordinary linear homogeneous differential equations in the neighborhood of an irregular singular point, \textit{Acta Mathematica} {\bf{93}} (1955) \rm 27--66

\bibitem{Was} {\sc W.~Wasow:}	Asymptotic expansions for ordinary differential equations, Wiley \& Sons (1965)


\end{thebibliography}
\end{document}